\documentclass[]{article}

\usepackage[english]{babel}
\usepackage{amsmath}
\usepackage{amsthm}
\usepackage{amsfonts}
\usepackage{epsf}
\usepackage{graphicx}

\newtheorem{theorem}{Theorem}
\newtheorem{corollary}[theorem]{Corollary}
\newtheorem{lemma}[theorem]{Lemma}
\newtheorem{definition}{Definition}

\numberwithin{theorem}{section}
\numberwithin{definition}{section}

\begin{document}
	
\title{Revisiting Taxicab Apollonius Circles}
\author{Kevin P. Thompson}
	
\date{}
	
\thispagestyle{empty}
\renewcommand\thispagestyle[1]{} 
	
\maketitle
	
\begin{abstract}
The existence of excircles and an Apollonius circle for a triangle in taxicab geometry are connected to the concept of inscribed triangles.
\end{abstract}
	
\section{Introduction}
	
A general problem in Euclidean geometry addressed by the ancient geometer Apollonius was the construction of a circle which is tangent to three other objects, typically points, lines, or circles.  If the other objects are taken as circles, a specific version of this problem can be stated as: “Is there the circle that tangentially encompasses all three excircles of a given triangle?”  The answer in Euclidean geometry is, “Yes.”  The answer in taxicab geometry is … “Sometimes.”

Apollonius circles for taxicab triangles were first investigated in \cite{Ermis}.  The approach was very direct and very detailed.  The existence of a taxicab Apollonius circle for a triangle was conditioned on the slopes of the lines forming the triangle.  The present work hopes to reformulate many of these results in a more elegant way by incorporating the concepts of inscribed angles and triangles presented earlier in \cite{Thompson}.  Along the way we will also show that the conditions for an Apollonius circle stated in \cite{Ermis} are equivalent to the results presented here with one minor exception and some other adjustments.

For convenience some definitions and results from  \cite{Thompson} are presented below with some minor additions necessary in later sections. It should be noted that this research is being done in pure taxicab geometry using native taxicab angles (see \cite{Trig}) and not modified taxicab geometry in which Euclidean angles are retained.

\section{Inscribed Angles and Triangles}

Consider an angle whose vertex is placed at the vertex of a taxicab circle.  A number of orientations and conditions can occur.  If the angle is fully contained in the circle, the angle is said to be \textbf{completely inscribed} (Figure \ref{fig:inscribed-angles}b).  If the angle straddles the side of the circle with slope -1 (Figure \ref{fig:inscribed-angles}a), then the angle is said to be \textbf{strictly positively inscribed} (positive since a line of slope +1 through the vertex remains outside the angle).  Similarly, an angle straddling the side of the circle with slope +1 (Figure \ref{fig:inscribed-angles}c) is \textbf{strictly negatively inscribed}.

\begin{figure}
	\centering
	\includegraphics[width=8cm]{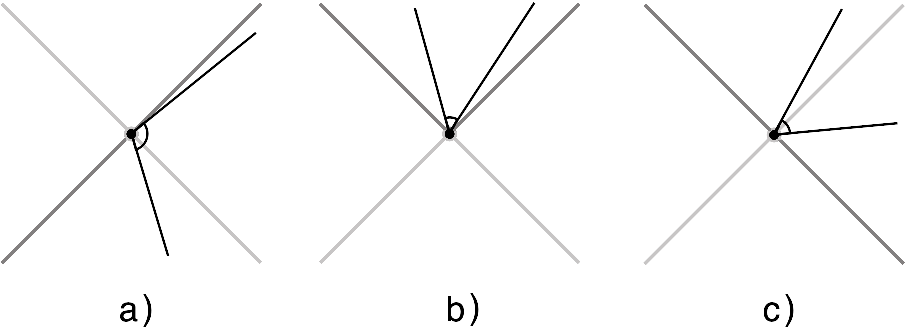}
	\caption{Inscribed taxicab angles. The angles shown are a) strictly positively inscribed, b) completely inscribed, and c) strictly negatively inscribed}
	\label{fig:inscribed-angles}
\end{figure}

In Euclidean geometry, using any point on a circle as the vertex, all angles less than $\pi$ radians can be represented as an inscribed angle by using two other points on the circle.  Since taxicab circles are composed of straight lines, they lack the curvature and “flexibility” of Euclidean circles. This causes some taxicab angles to not be inscribed.

It will be helpful to identify a primary characteristic of an angle that is not inscribed.  Without loss of generality, anchor one ray of an angle at the origin and place it in the upper half of Quadrant $I$ (the black ray in Figure \ref{fig:not-inscribed}).  The proposed second rays for the angle in light gray will form inscribed angles (using the dashed gray lines as guides, rays 3, 4, 5, or 6 form a negatively inscribed angle and rays 1, 2, 3, or 4 a positively inscribed angle).  The black ray in Quadrant $IV$ illustrates the kind necessary to form an angle that is not inscribed.  The important characteristic to note is that an angle that is not inscribed will always fully encompass a right angle aligned along lines of slope $\pm 1$, for otherwise one of the dashed lines would remain outside the angle therefore making it inscribed.

\begin{figure}[b]
	\centering
	\includegraphics{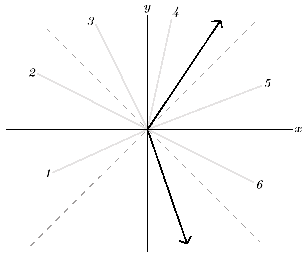}
	\caption{An angle that is not inscribed (black rays) opens so widely that it will always fully encompass a right angle aligned along lines of slope $\pm 1$}
	\label{fig:not-inscribed}
\end{figure}

Continuing with other helpful concepts, triangles with inscribed angles will be of particular interest.

\begin{definition}\label{def:inscribed-triangle}
A triangle is an \textbf{inscribed triangle} if all of its angles are inscribed.
\end{definition}

Curiously, the presence of even one strictly inscribed angle in a triangle is enough to force the triangle to be inscribed.

\begin{theorem}
If a triangle contains a strictly positively inscribed angle, then the other angles are negatively inscribed (and the triangle is therefore inscribed). 
\end{theorem}

\begin{corollary}\label{cor:strict-neg}
If a triangle contains a strictly negatively inscribed angle, then the other angles are positively inscribed (and the triangle is therefore inscribed).
\end{corollary}

Inscribed triangles also impose additional restrictions on their angles, and the more completely inscribed angles that are present the more limited the shape of the triangle.

\begin{theorem}
An inscribed triangle has at least one completely inscribed angle.
\end{theorem}

\begin{definition}\label{def:minimally-inscribed}
A triangle is \textbf{minimally inscribed} if it has only one completely inscribed angle.
\end{definition}

\begin{theorem}\label{thm:two-inscribed}
A inscribed triangle with exactly two completely inscribed angles has a side with slope $\pm 1$. 
\end{theorem}

\begin{theorem}\label{thm:three-inscribed}
An inscribed triangle with three completely inscribed angles has two sides with slope $\pm 1$. 
\end{theorem}

In Euclidean geometry "three points make a circle."  This is not always the case in taxicab geometry.  A crucial result for our investigation of Apollonius circles will be the Three-point Circle Theorem below.

\begin{theorem}[Three-point Circle Theorem] \label{thm:three-point}
Three non-collinear points lie on a taxicab circle if and only if the triangle they form is inscribed.
\end{theorem}

\begin{figure}[b!]
	\centering
	\includegraphics[width=\linewidth]{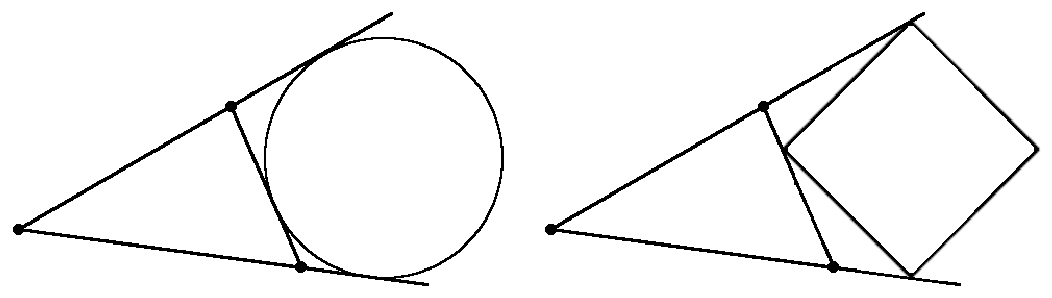}
	\caption{Excircles in Euclidean and taxicab geometry}
	\label{fig:excircles}
\end{figure}

\section{Taxicab Excircles}

There are two distinct steps in the process for determining if a triangle has a taxicab Apollonius circle: does an excircle  exist for each side of the triangle, and if so, does a circle exist that tangentially encompasses the three excircles?  We will answer these questions separately. Implicitly this was achieved in \cite{Ermis} but it will be instructive to clearly separate the conditions necessary for each step.  In addition, the concepts of inscribed angles and triangles are closely related to excircles and Apollonius circles, so it will be very instructive to also incorporate these concepts into the investigation. First, let's review the definition of an excircle. 

\begin{definition}\label{def:excircle}
An \textbf{excircle} (“escribed circle”) of a triangle is a circle outside the triangle that is tangent to a side of the triangle and tangent to the rays that form the angle opposite the side in the triangle.
\end{definition}

\begin{lemma}\label{lem:not-inscribed}
The side of a triangle opposite a taxicab angle that is not inscribed does not have an excircle.
\end{lemma}

\begin{proof}
Without loss of generality, anchor one ray of the angle at the origin and place it in the upper half of Quadrant $I$ (the black ray in Figure \ref{fig:not-inscribed-excircle}).  As explained in the previous section, an angle that is not inscribed will always fully encompass a right angle aligned along lines of slope $\pm 1$, so the black ray in Quadrant $IV$ is of the kind necessary in this situation.   When two points along the rays are connected to form the side of the triangle opposite the angle, any taxicab circle (green) tangent at its left point to the side of the triangle will never intersect the black rays because an angle that is not inscribed opens too widely.  This is true no matter how the rays are connected to form the side of the triangle.  The sides of the taxicab circle will simply not be steep enough to intersect either of the diverging rays to form an excircle.  In addition, any attempt to draw taxicab circles at the vertices will trivially touch one of the black rays but will fail to touch the other (imagine expanding the red taxicab circles).  Therefore, an excircle cannot be formed at the side of a triangle opposite an angle that is not inscribed.
\end{proof}

\begin{figure}[h!]
	\centering
	\includegraphics[width=6cm]{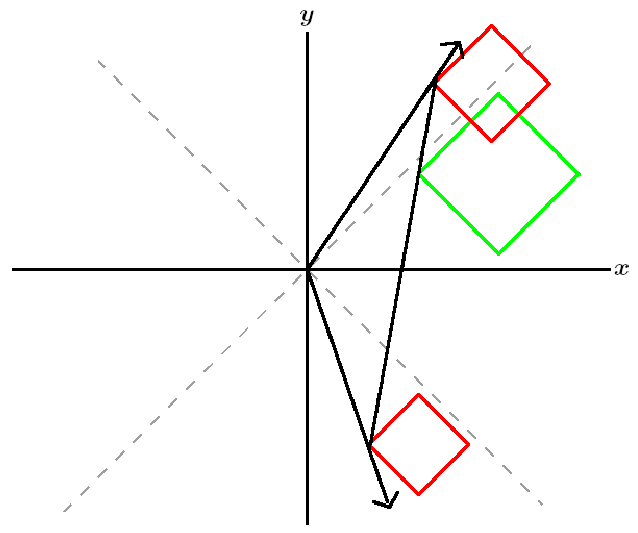}
	\caption{An angle that is not inscribed (black rays) opens too widely to allow the construction of an excircle on the side opposite the angle (green taxicab circle) or at the vertices (red taxicab circles)}
	\label{fig:not-inscribed-excircle}
\end{figure}

This lemma largely correlates to Theorem 2.1 of \cite{Ermis} but codifies the result in terms of inscribed angles. It indicates that the search for excircles and therefore Apollonius circles must be constrained to triangles whose angles are all inscribed.  This is the precise definition an inscribed triangle (Definition \ref{def:inscribed-triangle}).  We will examine the classes of inscribed triangles separately based on the number of completely inscribed angles in the triangle.  We will see that it is only possible to construct taxicab excircles on all sides of a triangle for certain kinds of inscribed triangles.

\begin{figure}
	\centering
	\includegraphics[width=5cm]{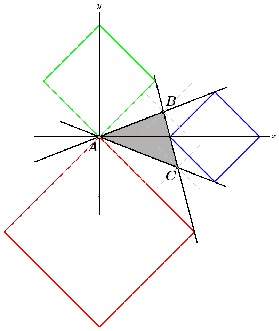}
	\caption{Constructing taxicab excircles (green, blue, and yellow) for a triangle with exactly one completely inscribed angle.}
	\label{fig:excircles-one-comp-insc}
\end{figure}

\begin{lemma}\label{lem:excircles-one-comp-insc}
Excircles can be constructed at all sides of a minimally inscribed taxicab triangle.
\begin{proof}
Let $\triangle ABC$ have a completely inscribed $\angle A$ with strictly negatively inscribed $\angle B$ and strictly positively inscribed $\angle C$ as shown in Figure \ref{fig:excircles-one-comp-insc}.  Since $\angle A$ is completely inscribed, in this case the absolute slope of its rays is less than 1; therefore, a taxicab circle in its interior (blue) can be expanded to be tangent to both of the angle’s rays.  Using the gray dashed lines at $\angle B$ and $\angle C$ we can see that the strictly inscribed nature of the angles limits the slope possibilities for $\overline{BC}$ in such a way that taxicab circles outside the triangle and tangent to $\overline{BC}$ will remain outside the triangle.  Therefore, $\overline{BC}$ has an excircle.

Taxicab circles tangent to the outside of $\overline{AB}$ can never be tangent to extensions of $\overline{AC}$ unless the bottom vertex of the taxicab circle is actually at point $A$, but such circles will always be able to be tangent to extensions of $\overline{BC}$ if of sufficient size (green).  Similarly, the excircle for $\overline{AC}$ must have its top vertex at point $A$ (red).

Triangles with other orientations of the completely inscribed angle can be proven similarly.
\end{proof}
\end{lemma}

\begin{lemma}\label{lem:excircles-two-comp-insc}
Excircles can be constructed at all sides of an inscribed taxicab triangle that has two completely inscribed angles only if for the two non-diagonal sides one side is steep and the other side is shallow.
\begin{proof}
By Theorem \ref{thm:two-inscribed}, triangles with two completely inscribed angles have one side with slope $\pm 1$; choose $\overline{BC}$ (i.e. $|m_a|=1$).

\noindent Case 1: $|m_b|<1$ and $|m_c|<1$ (left triangle in Figure \ref{fig:excircles-two-comp-insc}).  Any possible excircle placed along $\overline{AB}$ (red) will never intersect extensions of both $\overline{AC}$ and $\overline{BC}$ simultaneously.  Since the slope of the proposed excircle is greater than or equal to the slopes of both lines, only by placing the circle at a vertex can one side be touched, but then the other side will not be touched.  Therefore, $\overline{AB}$ cannot have an excircle tangent to it. The case of both sides steep is proven similarly.

\noindent Case 2: $|m_b|>1|$ and $|m_c|<1$ (right triangle in Figure \ref{fig:excircles-two-comp-insc}).  With $\overline{AC}$ now steep, a circle can be placed with its bottom vertex at $B$ and inflated to eventually intersect the extension of $\overline{AC}$ thus forming an excircle for $\overline{AB}$.  The excircle for $\overline{BC}$ exists due to its diagonal nature and the excircle for $\overline{AC}$ exists due to the shallow slope of $\overline{AB}$.						      
\end{proof}
\end{lemma}

\begin{figure}
	\centering
	\includegraphics[width=10cm]{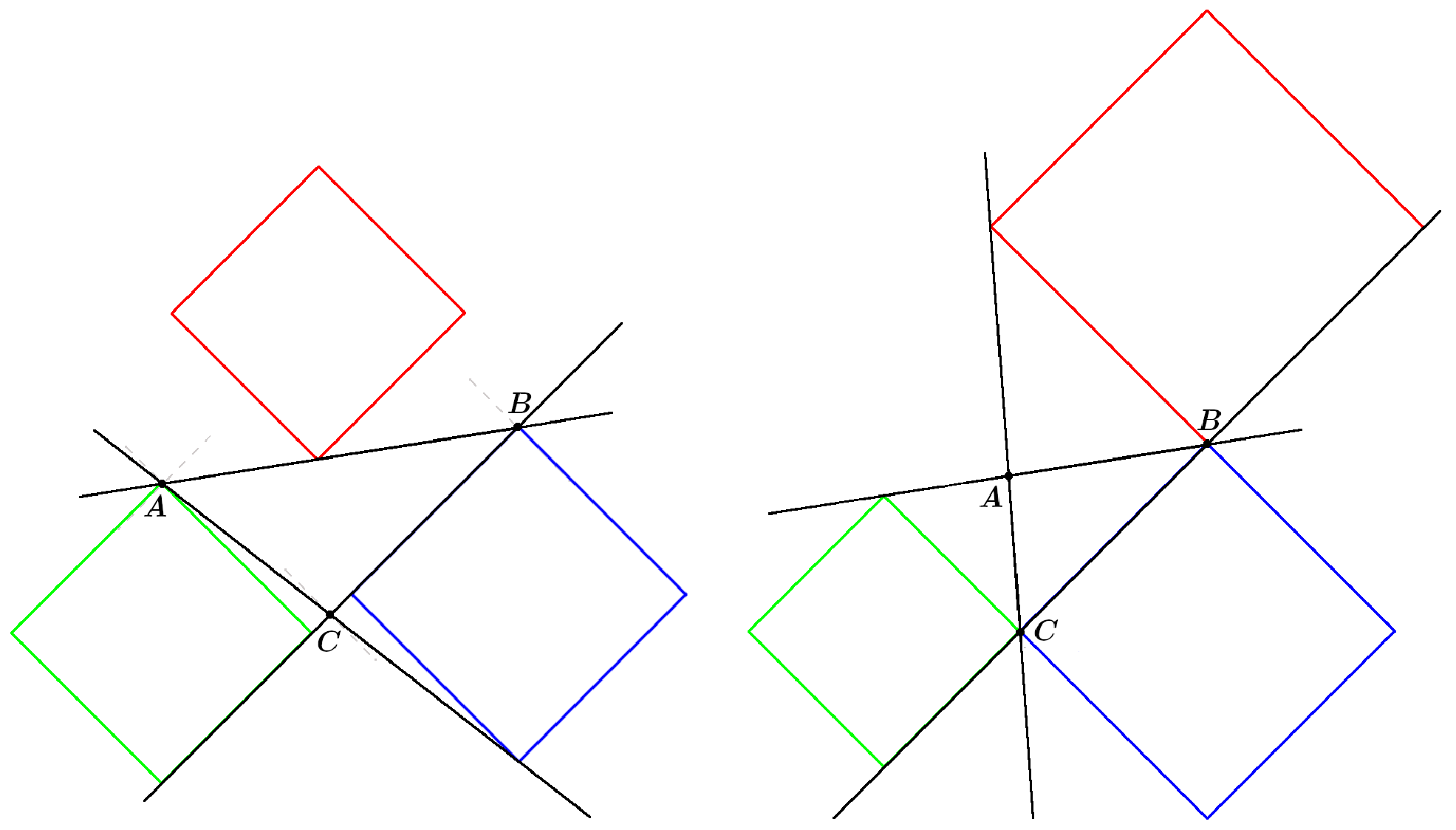}
	\caption{Attempting to construct taxicab excircles (green, blue, and red) for a triangle having two completely inscribed angles in a failing case (left) and a successful case (right).}
\label{fig:excircles-two-comp-insc}
\end{figure}

\begin{lemma}
Excircles cannot be constructed at all sides of an inscribed taxicab triangle that has three completely inscribed angles.
\begin{proof}
Two of the sides must have slope $\pm 1$ (Theorem \ref{thm:three-inscribed}). This type of triangle will suffer the same problem as the failing case of a triangle with two completely inscribed angles.  The left triangle in Figure \ref{fig:excircles-two-comp-insc} just needs to be subtly adjusted so that $\angle C$ is a right angle.  For the side opposite the right angle, the extensions of the triangle’s other sides will diverge away parallel to any possible excircle that could be drawn, and a circle at either vertex will only intersect one of the side extensions.  Therefore, the triangle will only have two excircles.	      
\end{proof}
\end{lemma}

\section{Taxicab Apollonius Circles}

While the search for taxicab excircles had a relatively clean result in terms of inscribed angles, the existence of an Apollonius circle for those excircles is fraught with an additional complication.  We will take a simple example of a triangle with a vertical side and a horizontal side to illustrate this complication.

\begin{figure}
	\centering
	\includegraphics[width=6cm]{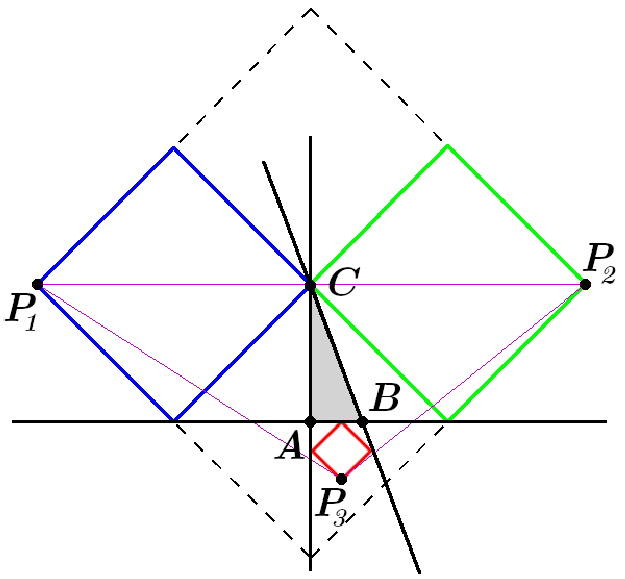}
	\caption{A triangle whose taxicab excircles do not have an enclosing tangential (Apollonius) circle}
	\label{fig:apollonius-hor-vert-fail}
\end{figure}

The triangle shown in Figure \ref{fig:apollonius-hor-vert-fail} fits the conditions of Lemma \ref{lem:excircles-one-comp-insc} and so it has a full complement of excircles.  But, the side $\overline{BC}$ is very steep.  This causes the excircle for side $\overline{AB}$ (red) to be squeezed between the extensions of $\overline{AC}$ and $\overline{BC}$ to the point that it is in a sense “interior” to the other two excircles (blue and green).  The effect of this is that any circle tangent to the excircles for $\overline{AC}$ and $\overline{BC}$ (dashed black line) will not be tangent to the excircle for $\overline{AB}$; the third small excircle will always be on the interior of the enclosing circle.

The technical reason that a taxicab Apollonius circle does not exist in this case is as follows.   Choose points $P_1$, $P_2$, and $P_3$, one on each excircle as shown.  Then $\angle P_1P_3P_2$ is not inscribed (it fully encompasses a right angle aligned along lines of slope $\pm 1$) and therefore $\triangle P_1P_3P_2$ (purple) is not inscribed.  By the Three-Point Circle Theorem (Theorem \ref{thm:three-point}) there is no taxicab circle passing through these three points.

If the slope of $\overline{BC}$ is reduced sufficiently, the enclosing circle can be made tangent (Figure  \ref{fig:apollonius-success}).  But, if the slope is reduced to unity, then the triangle will have three inscribed angles and therefore fail to have a full complement of excircles.  So, there is a bounded range for the slope of $\overline{BC}$ in which an Apollonius circle exists.

\begin{figure}
	\centering
	\includegraphics[width=6cm]{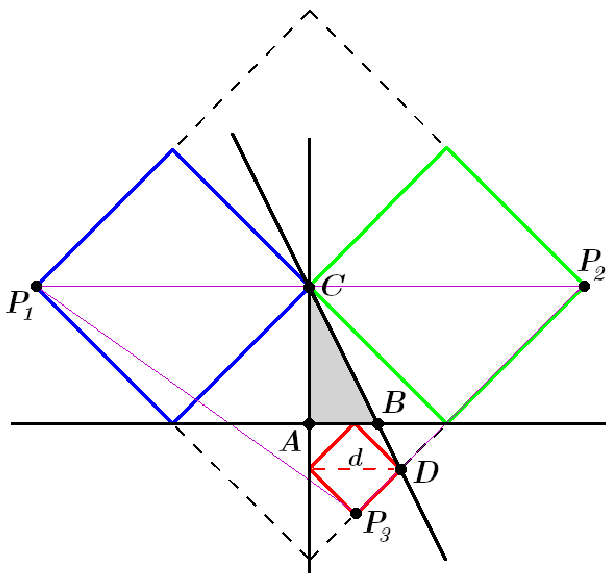}
	\caption{A triangle whose taxicab excircles (barely) have an enclosing tangential (Apollonius) circle}
	\label{fig:apollonius-success}
\end{figure}

To find the slope range for $\overline{BC}$, notice that the point $D$ in Figure \ref{fig:apollonius-success} is the crucial point.  In order for $\angle P_1P_3P_2$ to be inscribed (so that the points form an inscribed triangle and the Three-Point Circle Theorem can be applied to construct the Apollonius circle), $D=\left(d,-\frac{1}{2}d\right)$ must lie on the extension of $\overline{BC}$ and \textit{on or to the right of} the extension of the lower right section of the excircle of side $\overline{BC}$ (green).  Using the notation $(x_*,y_*)$ for the coordinates of the vertices of the triangle, the equation for the extension of the excircle for side $\overline{BC}$ (green) is $y=x-y_c$ which using the point $D$ on $\overline{BC}$ leads to the constraint $-\frac{1}{2} d \le d-y_c$ or $y_c \le \frac{3}{2}d$.

The equation for the line through $\overline{BC}$ with slope $m_a$ is $y=-m_ax+y_c$ which for point $D$ gives $-\frac{1}{2}d = m_a d + y_c$. When combined with the inequality above leads to the constraint
\begin{align}
	y_c = -(m_a+ \frac{1}{2})d & \le \frac{3}{2}d \nonumber \\
	m_a & \ge -2 \nonumber
\end{align}

Therefore, in this orientation, the slope of $\overline{BC}$ is constrained to $-2\le m_a < -1$ in order for the taxicab Apollonius circle to exist.  This analysis provides a template for proving all four cases (for the hypotenuse slope to be steep vs. shallow and positive vs. negative) of the following lemma and also outlines the kind of strategy necessary for handling other types of triangles leading up to a general Triangle Apollonius Circle Theorem.  Note that this result agrees with Corollary 2.1 in \cite{Ermis}.

\begin{lemma}\label{lem:apollonius-hor-vert}
A triangle with a horizontal side and a vertical side has an Apollonius circle tangentially encompassing its taxicab excircles if the slope $m$ of the side opposite the right angle satisfies $\frac{1}{2} \le |m| < 1$ or $1<|m|\le 2$.
\end{lemma}

Before dealing with triangles having other orientations and shapes, it should be demonstrated that the conclusions in \cite{Ermis} are not quite correct.  Consider the triangle formed by lines of slope $\frac{3}{4}$, $\frac{1}{5}$, and 3 shown with its Apollonius circle (orange) in Figure \ref{fig:ermis-fail}.  The value of $\rho$ in Theorem 2.2 of \cite{Ermis} for this triangle is either -0.1 or -10 depending on the choice of A and B, which in and of itself poses a problem since the theorem poses no restrictions on the choice of $A$ and $B$ yet the choice changes the result.  But furthermore, the triangle satisfies branch 2 of part i) of this theorem which indicates that $\rho$ must be positive in order for the taxicab Apollonius circle to exist.

\begin{figure}[t!]
	\centering
	\includegraphics[width=6cm]{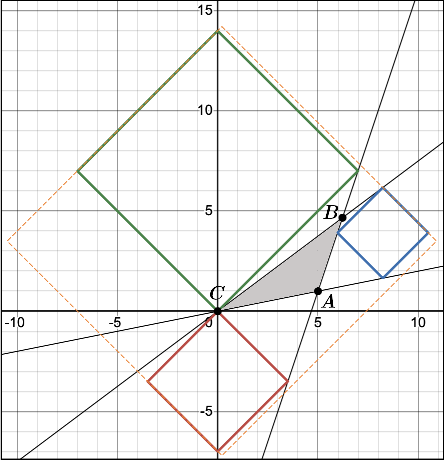}
	\caption{Example of a triangle with a taxicab Apollonius circle (orange) that violates the conditions of Theorem 2.2 of \cite{Ermis}}
	\label{fig:ermis-fail}
\end{figure}

We will now attempt to find an Apollonius circle for a more general triangle.  Consider a triangle formed by two shallow lines with same signed slope as in Figure \ref{fig:apollonius-one-comp-insc}, assuming without loss of generality that $0\le m_b<m_a<1$.  The $\angle C$ formed at the intersection of the lines is completely inscribed.

Before continuing, consider any triangle with one completely inscribed angle formed by two lines whose slopes are both steep or both shallow and have the same sign.  Such a triangle is always equivalent to a triangle positioned in the lower half of the first quadrant.  If the axes are relabeled, or equivalently if the triangle is reflected about the lines $y=\pm x$ and/or $x=0$ and/or $y=0$ and possibly translated, the triangle can be “moved” to the lower half of the first quadrant with the completely inscribed angle at the origin.  Therefore, our analysis will be limited to the possible cases in the lower half of the first quadrant as in Figure \ref{fig:apollonius-one-comp-insc}.

In order for excircles to exist, there are constraints on the slope $m_c$ of side $\overline{AB}$.  If $0<m_c<1$ (purple dashed line) then $\angle B$ will be a second completely inscribed angle but the triangle will not satisfy Lemma \ref{lem:excircles-two-comp-insc} and therefore will not have a full complement of excircles.  If $-1<m_c<0$ (other purple dashed line) then $\angle A$ is again a second completely inscribed angle, but the triangle will still not satisfy Lemma \ref{lem:excircles-two-comp-insc} and will not have a full complement of excircles. So, we must in the end include the restriction $|m_c|>1$. Note that the triangle is minimally inscribed.

\begin{figure}
	\centering
	\includegraphics[width=8cm]{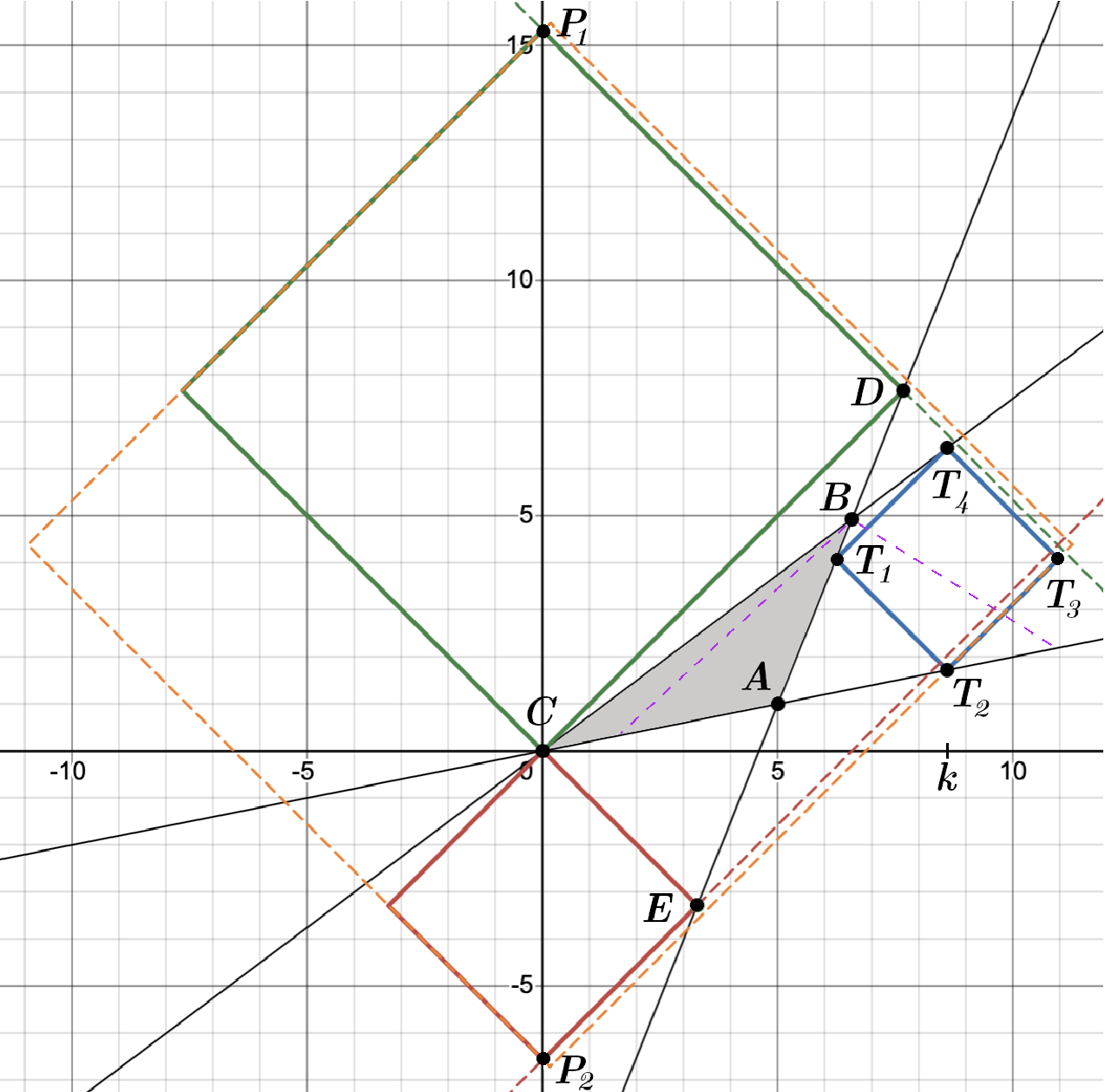}
	\caption{Constructing excircles for a triangle contained in one quadrant with one completely inscribed angle and a positive steep-sloped opposite side.}
	\label{fig:apollonius-one-comp-insc}
\end{figure}

Using Figure \ref{fig:apollonius-one-comp-insc} as a guide, the radius of the excircle for $\overline{AB}$ is $r=\frac{1}{2}k(m_a-m_b)$ and so the coordinates of the vertices of the excircle are:
\begin{align}
	T_2 & =(k,m_b k) \nonumber \\
	T_4 & =(k,m_a k) \nonumber \\
	T_1 & = (k-r,m_ak-r) = \left(k+\frac{k}{2}(m_a+m_b),\frac{k}{2}\left(m_a+m_b\right)\right) \nonumber \\
	T_3 & = (k+r,m_ak-r) = \left(k+\frac{k}{2}(m_a-m_b),\frac{k}{2}\left(m_a+m_b\right)\right) \nonumber
\end{align}

\noindent Since $T_1$ is on the side $\overline{AB}$ which has equation $y=m_c(x-x_a)+y_a$, we can substitute its coordinates and solve for $k$ to get

\[
k=\frac{2(y_a-x_am_c)}{m_a+m_b+m_am_c-m_bm_c-2m_c}
\]

Now, $\angle P_1T_3P_2$ must be inscribed in order for the three points to lie on a taxicab circle (Three-Point Circle Theorem).  In order for this to be the case, $T_3$ cannot be interior to the corner formed by the red and green extensions of the excircles for $\overline{AC}$ and $\overline{BC}$, meaning $T_3$ is either above the green dashed line or below the red dashed line for $\angle P_1T_3P_2$ to be inscribed.\footnote{Theorem 2.2 in \cite{Ermis} mistakenly requires both conditions.}

Taking individually the constraints that will yield an Apollonius circle, first consider the green dashed line.  Point $D$ is the intersection of  $y=x$ with the extension of $\overline{AB}$ with equation $y=m_c(x-x_a)+y_a$.  This yields

\[
D=\left(\frac{y_a-x_a m_c}{1-m_c},\frac{y_a-x_am_c}{1-m_c}\right)
\]

\noindent So, the equation of the green dashed line is

\[
y=-x+\frac{2(y_a-x_a m_c)}{1-m_c}
\]

\noindent For $T_3$ to be above the green dashed line, $T_4$ would have to be as well.  Using this point for its simpler coordinates, its $y$-value must be greater than the equation for the green line: 
\[
m_a k\ge -k+\frac{2(y_a-x_am_c)}{1-m_c}
\]

\noindent Noting that $m_a+1>0$, solving for $k$ gives
\[
k\ge \frac{2(y_a-x_am_c)}{(1-m_c)(m_a+1)}
\]

\noindent Substituting for $k$:
\[
\frac{2(y_a-x_am_c)}{m_a+m_b+m_am_c-m_bm_c-2m_c}\ge \frac{2(y_a-x_am_c)}{(1-m_c)(m_a+1)}
\]

\noindent Since $k>0$ in this situation, the left denominator must have the same sign as the left numerator.  For $m_c>1$ all fraction parts are negative; for $m_c<-1$ all fraction parts are positive.  Taking this into account, after simplification we have:
\begin{align}
	m_b+m_c(2m_a-m_b-1) \ge 1, & \,\,\, m_c>1 \nonumber \\
	m_b+m_c(2m_a-m_b-1) \le 1, & \,\,\, m_c<-1 \nonumber
\end{align}

\noindent Dividing through either inequality by $m_c$ yields one condition that a triangle with $0\le m_b<m_a<1$ can satisfy in order to have a taxicab Apollonius circle:
\[
\frac{m_b}{m_c}+2m_a-m_b-1 \ge \frac{1}{m_c}
\]

For the alternate constraint, point $T_2$ would have to be below the red dashed line.  The point $E$ is the intersection of $y = -x$ with $y=m_c(x-x_a)+y_a$ which yields
\[
E=\left(\frac{x_a m_c-y_a}{m_c+1},\frac{y_a-x_a m_c}{m_c+1}\right)
\]

\noindent The equation of the red dashed line is therefore
\[
y=x+\frac{2(y_a-x_a m_c)}{m_c+1}
\]

\noindent So, with $T_2$ below the line we have
\[
m_b k\le k+\frac{2(y_a-x_am_c)}{m_c+1}
\]

\noindent Solving for $k$ ($m_b-1<0$) and substituting gives
\[
\frac{2(y_a-x_am_c)}{m_a+m_b+m_am_c-m_bm_c-2m_c}\ge \frac{2(y_a-x_am_c)}{(m_c+1)(m_b-1)}
\]

\noindent The sign analysis has not changed so simplifying yields
\[
1-m_a+2m_b-\frac{m_a}{m_c} \le \frac{1}{m_c}
\]

\noindent This is the second condition that a triangle with $0\le m_b < m_a < 1$ can satisfy in order to have a taxicab Apollonius circle.  Note that combining the two constraints yields the overly restrictive function $\rho$ in Theorem 2.2 of \cite{Ermis}.

As an example, the triangle in Figure \ref{fig:apollonius-one-comp-insc} has slopes $\frac{3}{4}$, $\frac{1}{5}$, and $\frac{5}{2}$.  It satisfies the second constraint ($0.35 \le 0.4$) but not the first ($0.38 \not\ge 0.4$) which confirms what is visually seen in the figure.

For a minimally inscribed triangle, the only remaining case is when the signs of $m_a$ and $m_b$ differ, namely $-1<m_b<0<m_a<1$.  But, even in this case the items of consequence ($m_a+1>0$, $m_b-1<0$, and $k>0$) in the analysis still hold up even if a relabeling of the points must occur.  Therefore, we have the following result.

\begin{lemma}\label{lem:apollonius-one-comp-insc}
A minimally inscribed triangle with shallow slope sides $0<m_a<1$ and $-1<m_b<m_a$ forming the completely inscribed $\angle C$ has an Apollonius circle tangentially encompassing its taxicab excircles if one of the following conditions holds:
\begin{enumerate}
\item for $0<|m_c|<\infty$: $\frac{m_b}{m_c}+2m_a-m_b-1 \ge \frac{1}{m_c}$ or $1-m_a+2m_b-\frac{m_a}{m_c} \le \frac{1}{m_c}$
\item for $\overline{AB}$ vertical: $2m_a-m_b \ge 1$
\end{enumerate}

\noindent where $m_a$ and $m_b$ are the slopes of the sides opposite $\angle A$ and $\angle B$, respectively.
\end{lemma}

To establish the constraint for the $\overline{AB}$ vertical case let $m_c \to \pm \infty$ in the non-vertical case. Then the first constraint simplifies to $2m_a-m_b \ge 1$.  Also note that Lemma \ref{lem:apollonius-hor-vert} is a special case of Lemma \ref{lem:apollonius-one-comp-insc} when limited to the first quadrant.

For a triangle with two completely inscribed angles, we already know that the triangle must have one side with slope $\pm 1$ (Theorem \ref{thm:two-inscribed}) and by Lemma \ref{lem:excircles-two-comp-insc} the other sides must have opposing slopes - one shallow and one steep.  Assume $|m_b|<1$ and $|m_c|>1$.  The basic constraint process outlined above for minimally inscribed triangles works here as well.  If we simply apply the first constraint of Lemma \ref{lem:apollonius-one-comp-insc} with $m_a=1$ we find that $m_c \ge 1$ (for $m_c>1$) or $m_c \le 1$ (for $m_c<-1$) which we have already required in order to even have a full complement of excircles.  Visually, in Figure \ref{fig:apollonius-two-comp-insc}, the side $\overline{AB}$ can have any slope $|m_c|>1$ and the excircles for $\overline{AB}$ (blue) and $\overline{BC}$ (green) are unaffected leaving the excircle for $\overline{AB}$ always to the right of the extension of the excircle for $\overline{BC}$, which is exactly what the first constraint of Lemma \ref{lem:apollonius-one-comp-insc} is testing.  Therefore, a triangle with two completely inscribed angles that has a full complement of excircles will always have a taxicab Apollonius circle.

\begin{figure}
	\centering
	\includegraphics[width=6cm]{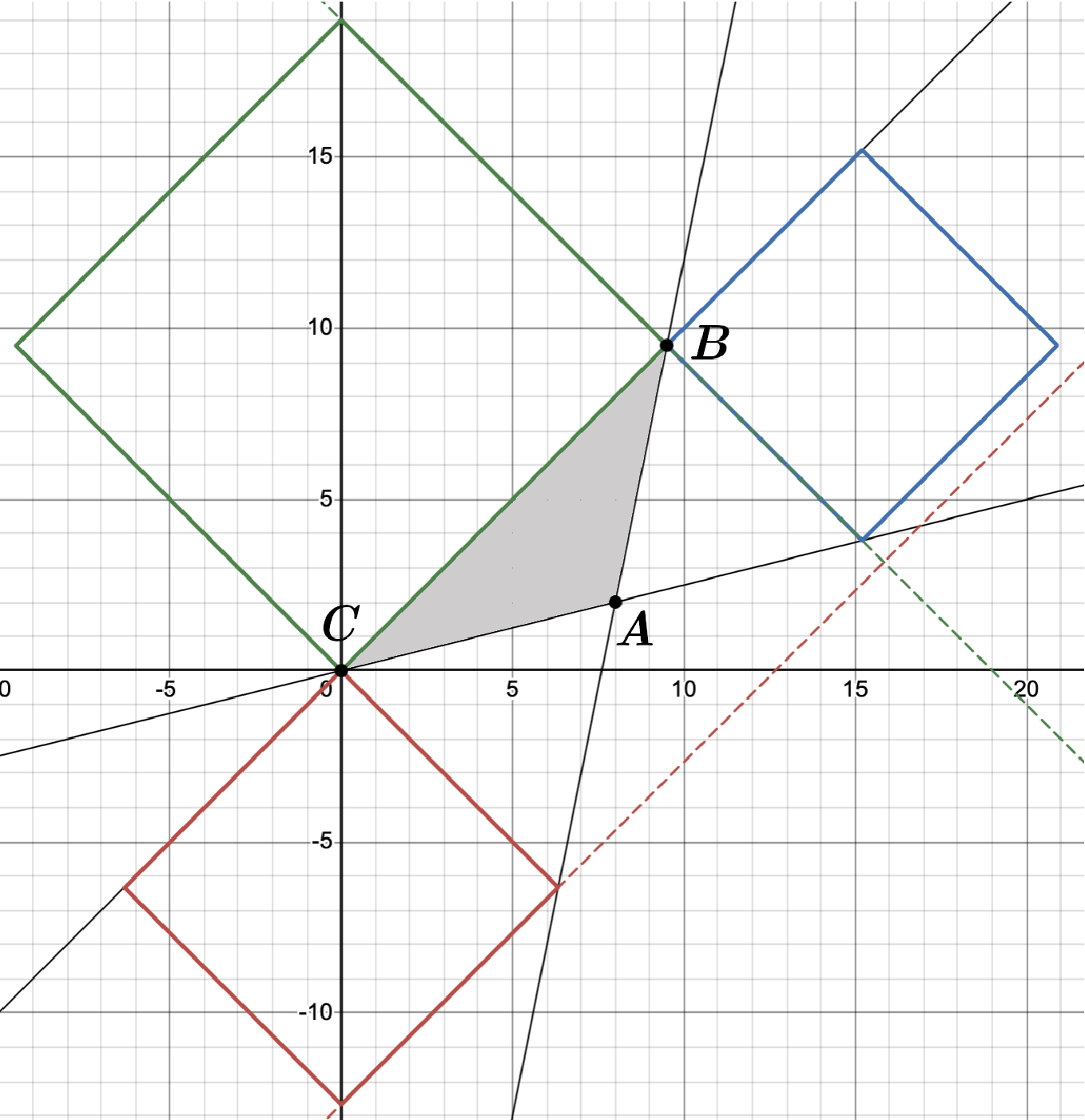}
	\caption{Excircles for a triangle with two completely inscribed angles}
	\label{fig:apollonius-two-comp-insc}
\end{figure}

\begin{theorem}[Triangle Apollonius Circle Theorem]
A $\triangle ABC$ has a taxicab Apollonius circle if it is transformationally equivalent to one of the following triangles:
\begin{enumerate}
\item a minimally inscribed triangle with shallow slope sides $0<m_a<1$ and $-1<m_b<m_a$ forming the completely inscribed $\angle C$ with either
\begin{enumerate}
\item $1<|m_c|<\infty$ and either $\frac{m_b}{m_c}+2m_a-m_b-1 \ge \frac{1}{m_c}$ or $1-m_a+2m_b-\frac{m_a}{m_c} \le \frac{1}{m_c}$, or
\item $\overline{AB}$ vertical and $2m_a-m_b \ge 1$
\end{enumerate}
\item a triangle with exactly two completely inscribed angles and non-diagonal sides with opposing slopes (one shallow, one steep)
\end{enumerate}
\end{theorem}

\begin{figure}[h]
	\centering
	\includegraphics[width=5.5cm]{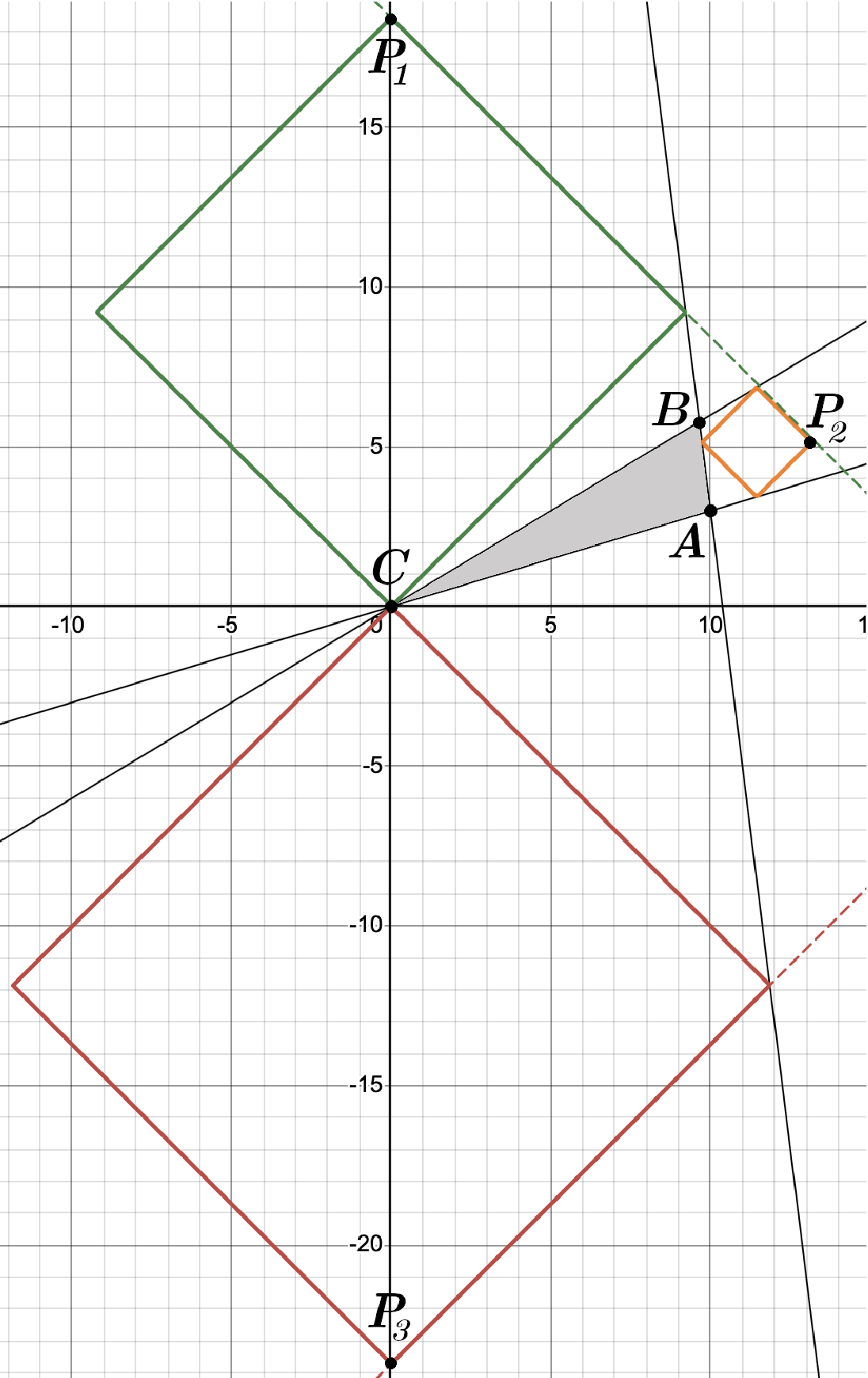}
	\caption{A triangle with a full complement of taxicab excircles that fails to have a taxicab Apollonius circle}
	\label{fig:apollonius-fail}
\end{figure}

Before concluding, let’s examine a general case where a taxicab Apollonius circle fails to exist. The triangle in Figure \ref{fig:apollonius-fail} has $m_a=\frac{3}{5}$, $m_b=\frac{3}{10}$, and $m_c=-8$.  It satisfies neither the first constraint ($-0.1375 \not\ge -0.125$) where the excircle for $\overline{AB}$ is narrowly below the dashed green line nor the second ($1.075 \not\le -0.125$) where is it far above the dashed red line.  The points $P_1$, $P_2$, and $P_3$ fail to form an inscribed triangle and so a taxicab circle through the points does not exist.

Other interesting questions can be asked such as: for which values of $m_a$ and $m_b$ will any choice of $m_c$ lead to a triangle with a taxicab Apollonius circle, or a lack of an Apollonius circle?  Solving the first and second constraints for $m_c>1$ in Lemma \ref{lem:apollonius-one-comp-insc} and taking care to consider only $2m_a-m_b-1>0$ for the moment, we get 
\begin{center}
$m_c\ge \frac{1-m_b}{2m_a-m_b-1}$ and $m_c \le \frac{m_a+1}{1-m_a+2m_b}$
\end{center}

\noindent If we wish to find where there is an Apollonius circle for any choice of $m_c$ then we would want
\[
\frac{1-m_b}{2m_a-m_b-1} < \frac{m_a+1}{1-m_a+2m_b}
\]

\noindent When this inequality is graphed with $2m_a-m_b-1>0$ the gray shaded region in Figure \ref{fig:apollonius-region} results.  The point marked is for the slopes from Figure \ref{fig:apollonius-one-comp-insc} which is just inside the shaded region.  This illustrates that we could have chosen any value for $|m_c|>1$ and the triangle would have had a taxicab Apollonius circle.

\begin{figure}
	\centering
	\includegraphics[width=7cm]{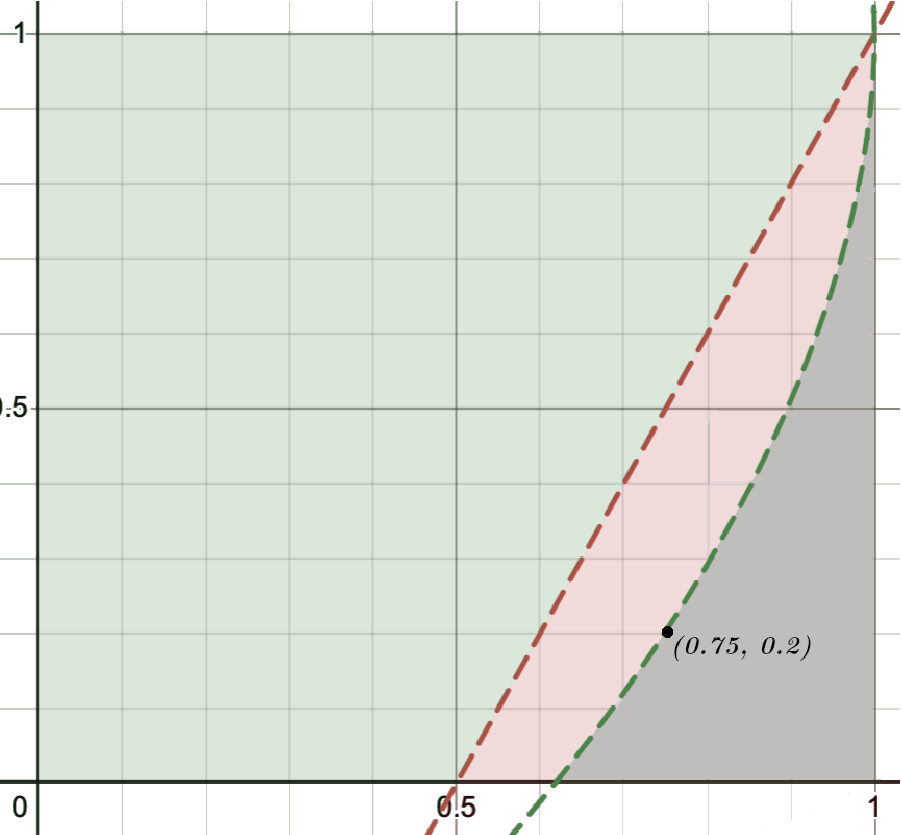}
	\caption{A region (gray) of values for $m_a$ ($x$-axis) and $m_b$ ($y$-axis) in which any choice of $|m_c|>1$ would lead to an Apollonius circle for the triangle}
	\label{fig:apollonius-region}
\end{figure}

As few final notes.

\subsection{Apollonius Points}
In Euclidean geometry, the Apollonius circle is tangent to each excircle at a point.  If these points are connected to the vertex associated with their excircle, the three lines intersect at a point called the Apollonius point.  Since the taxicab Apollonius circle is tangent to most excircles along a line segment the concept of an Apollonius point does not exist in taxicab geometry.

\subsection{Excenters}
If the centers of the excircles of a triangle (the excenters) are connected to their corresponding vertices, the three lines appear to be concurrent (Figure \ref{fig:excenter-lines-concurrent}).

\begin{figure}[t]
\centering
\includegraphics[width=7cm]{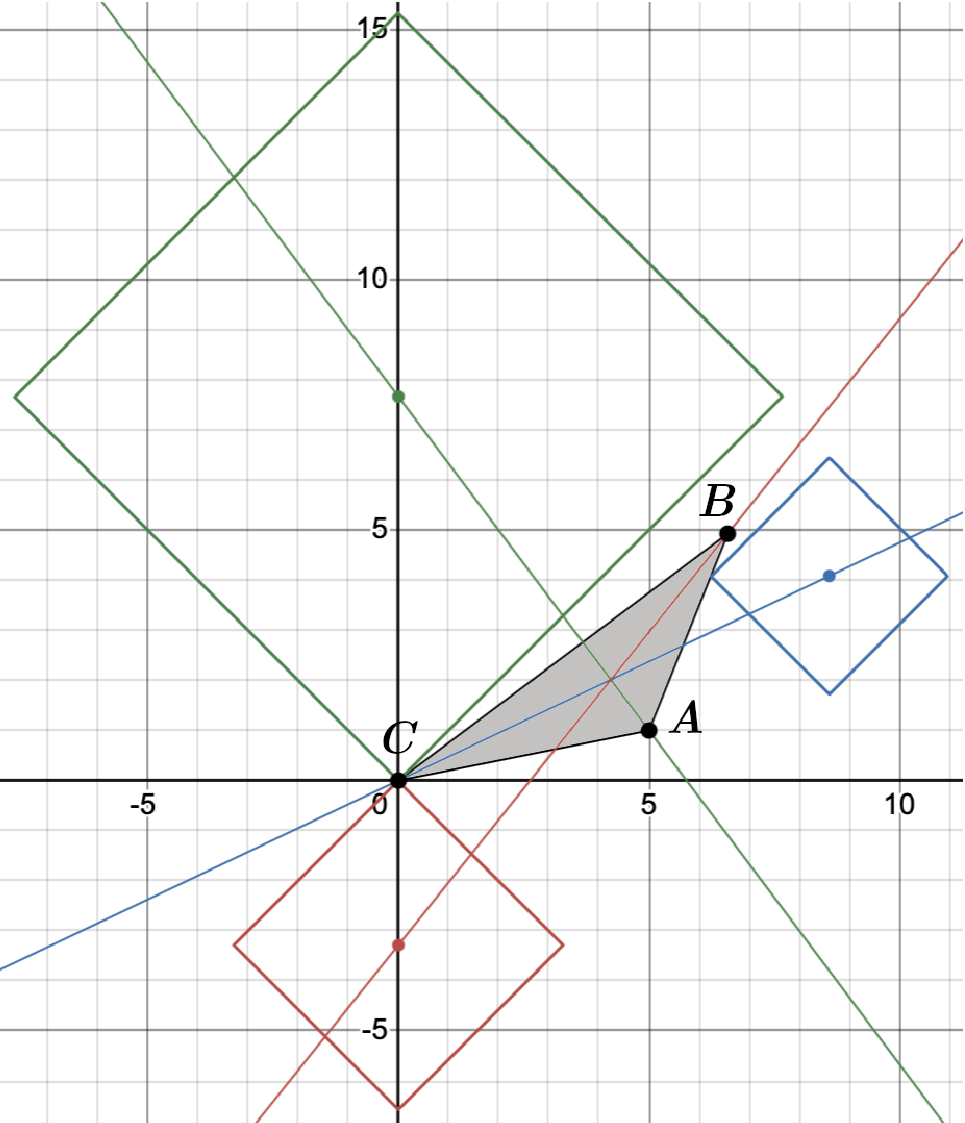}
\caption{Lines connecting the excenters of a triangle to their corresponding vertices are concurrent}
\label{fig:excenter-lines-concurrent}
\end{figure}


\begin{thebibliography}{3}
\bibitem{Ermis}
Ermis, Temel, Gelisgen, Ozcan, Ekici, Aybuke: {\em A Taxicab Version of a Triangle’s Apollonius Circle}. Journal of Mahani Mathematical Research Center. 7 (1), 25-36 (2018).

\bibitem{Trig}
Thompson, Kevin Dray, Tevian: {\em Taxicab Angles and Trigonometry}. The Pi Mu Epsilon Journal. 11 (2), 87-96 (2000).

\bibitem{Thompson}
Thompson, Kevin P.: {\em Taxicab Triangle Incircles and Circumcircles}. The Pi Mu Epsilon Journal. 13 (5), 299-305 (2011).
\end{thebibliography}
\end{document}